\documentclass[twoside]{amsart}
\usepackage{amscd}
\usepackage[all]{xy}

\usepackage{graphicx}

\usepackage{xcolor}

\newcommand{\rank}{\operatorname{rank}}

\newcommand{\C}{\mathbf{C}}
\renewcommand{\L}{\mathbf{L}}
\renewcommand{\P}{\mathbf{P}}
\newcommand{\N}{\mathbf{N}}
\newcommand{\Q}{\mathbf{Q}}
\newcommand{\Z}{\mathbf{Z}}
\newcommand{\un}{\mathbf{1}}

\newcommand{\sM}{\mathcal{M}}

 \numberwithin{equation}{section}

\theoremstyle{plain}
\newtheorem{thm}[equation]{Theorem}

\newtheorem{prop}[equation]{Proposition}

\newtheorem{cor}[equation]{Corollary}
\newtheorem{conj}[equation]{Conjecture}

\theoremstyle{definition}
\newtheorem{defn}[equation]{Definition}

\newtheorem{rk}[equation]{Remark}

\begin{document}

\title{A family of special cubic fourfolds with motive of abelian type}

\author[H. Awada]{Hanine Awada}
\address{Institut Montpellierain Alexander Grothendieck \\ %
Universit\'e de Montpellier \\ %
Case Courrier 051 - Place Eug\`ene Bataillon \\ %
34095 Montpellier Cedex 5 \\ %
France}
\email{hanine.awada@umontpellier.fr}

\author[M. Bolognesi]{Michele Bolognesi}
\address{Institut Montpellierain Alexander Grothendieck \\ %
Universit\'e de Montpellier \\ %
Case Courrier 051 - Place Eug\`ene Bataillon \\ %
34095 Montpellier Cedex 5 \\ %
France}
\email{michele.bolognesi@umontpellier.fr}

\author[C. Pedrini]{Claudio Pedrini}
\address{Dipartimento di Matematica \\ %
Universit\'a degli Studi di Genova \\ %
Via Dodecaneso 35 \\ %
16146 Genova \\ %
Italy}
\email{pedrini@dima.unige.it}

\begin{abstract} 

In this short note, we show that there exist one dimensional families of cubic fourfolds with Chow motive of abelian type and finite dimensional inside every Hassett divisor of special cubic fourfolds. This also implies abelianity and finite dimensionality of the motive of related Hyperk\"ahler varieties, such as the Fano variety of lines and the LLSvS 8fold.

\end{abstract}

\maketitle 

\section{Introduction}

Let $\sM_{rat}(\C)$ be the (covariant) category of Chow motives and let $\sM^{Ab}_{rat}(\C)$ be the strictly full, thick, rigid, tensor subcategory of $\sM_{rat}(\C)$ generated by the motives of abelian varieties. The following classes of smooth projective varieties are known to have motives belonging to $\sM^{Ab}_{rat}(\C)$:

\begin{enumerate}

\item projective spaces, Grassmannian varieties, projective homogeneous
varieties, toric varieties ;
\item smooth projective curves ;
\item Kummer K3 surfaces ;
\item K3 surfaces with Picard numbers at least 19;
\item K3 surfaces  with a non-symplectic group of automorphisms  acting trivially on the algebraic cycles: K3 surfaces satisfying these conditions have 
Picard numbers equal to $ 2, 4, 6, 10, 12, 16, 18, 20$, see \cite{Ped};
\item Hilbert schemes of points on  abelian surfaces;
\item Fermat hypersurfaces ;
\item Cubic 3-folds and their Fano surfaces of lines, see \cite{GG} and  \cite{Diaz}.
\end{enumerate}

In this short note we consider the case of cubic fourfolds. It is well known that to certain classes of cubic fourfolds one can append an \it associated K3 surface \rm, whose primitive cohomology is essentially the same as the non-special cohomology of the cubic fourfold. Moreover the moduli space of cubic fourfolds contains a countable infinity of divisors that parametrize cubics whose lattice of algebraic 2-cycles has rank bigger than expected \cite{Ha1}. These divisors are commonly called Hassett divisors of special cubic fourfolds. 

\medskip

On the other hand, very few examples of cubic fourfolds belonging to the category $\sM^{Ab}_{rat}(\C)$ are known, and they are often related to the Fermat cubic fourfold \cite{lat:somecub,lat:afamily,lat:aremarkfano}. Moreover, it is not clear how the known examples are positioned in the geography of the moduli space. Note that all the examples of motives that have been proven to be finite-dimensional, in the sense of Kimura-O'Sullivan, belong to the category $\sM^{Ab}_{rat}$.

\medskip

In this note, we try to start to fill these gaps, by constructing new one dimensional families of cubic fourfolds with finite dimensional, Abelian motive. Moreover, these families are quite ubiquitous in the moduli space, since any Hassett divisor contains some of them.

The proof is basically the combination of three ingredients. The first is the intersection theoretical machinery developed by the first named author in \cite{awada2020rational}, that allowed her to construct one dimensional families of cubic fourfolds with associated K3 surfaces with rank 19. These families can be arbitrarly constructed inside any Hassett divisor. Then we combine this with results of the second and third named authors  \cite{BP20} and B\"ulles \cite{Bu}, about the Chow-K\"unneth decomposition of the Chow motive of a cubic fourfold $X$. These allow us to reduce the finite dimensionality and Abelianity of $h(X)$ to the same properties of the associated K3 surface. Finally we apply results of the third named author \cite{Ped} to obtain our main theorem

\begin{thm}\label{main}
Every Hassett divisor $\mathcal{C}_d$ contains a countable infinity of one dimensional families of cubic fourfolds, whose Chow motive is finite dimensional and Abelian.
\end{thm}

\begin{rk}
We observe that our construction allows us to construct some families, with the properties as in Thm.\ref{main}, whose members are all rational. If Kuznetsov's conjecture \cite{Kuznetsov_2009} about the rationality of cubic fourfolds holds true, all the family we can construct would be made up of rational cubic fourfolds.
\end{rk}

\begin{rk}
It is not hard to show, and we do so in Sect. 4, that cubic fourfolds with finite dimensional motive are dense for the complex topology inside Hassett divisors.
\end{rk}

As a consequence of our main theorem we also observe the finite dimensionality and Abelianity of the motive of two related Hyperk\"ahler varieties, namely the Fano variety $F(X)$ of lines contained in a cubic fourfold $X$, and the Lehn-Lehn-Sorger-van Straten 8fold $L(X)$ constructed in \cite{LLSvS}. 

\subsection{Plan of the paper}

In Section 2 we recall some basic facts about the moduli of cubic fourfolds, and in Section 3 the main results from the theory of Chow motives that are needed. Section 4 contains the proof of the main Theorem and the results about Hyperk\"ahler varieties.

\section{The moduli space of cubic fourfolds}


\subsection{Special cubic fourfolds}

A cubic fourfold $X$ is a smooth complex cubic hypersurface in $\P^5$. The coarse moduli space of cubic fourfolds $\mathcal{C}$ is a 20-dimensional quasi-projective variety. It can be described as a GIT quotient $\mathcal{C}:= \mathcal{U} // PGL(6,\C)$, where $\mathcal{U}$ is the Zariski open subset of $\vert \mathcal{O}_{\P^5}(3) \vert$ parametrizing smooth cubic hypersurfaces in $\P^5$.

\medskip

The cohomology of cubic fourfolds is torsion-free and the middle cohomology $H^4(X,\Z)$ is the one containing nontrivial information (see \cite{Beau}) about the geometry of $X$. Hassett (\cite{Ha1}, \cite{Ha2}) studied cubic fourfolds via Hodge theory and introduced the notion of \textit{special} cubic fourfolds, corresponding to those containing an algebraic surface not homologous to a multiple of $h^2$, where $h$ is the class of a hyperplane section. Let $A(X)=H^4(X,\Z) \cap H^{2,2}(X)$ be the positive definite lattice of integral middle Hodge classes. Since the integral Hodge conjecture holds for cubic fourfolds (see \cite{Voi}), $X$ is \textit{special} if and only if the rank of $A(X)$ is at least 2. 

\begin{defn}
A labelling of a \textit{special} cubic fourfold is a rank 2 saturated sublattice $K \subseteq A(X)$ containing $h^2$. Its discriminant $d$ is the determinant of the intersection form on $K$.
\end{defn}

Hassett defined a countably infinite union of divisors $\mathcal{C}_d \subset \mathcal{C}$ parametrizing \textit{special} cubic fourfolds with labelling of discriminant $d$. He showed that $\mathcal{C}_d$ is irreducible and nonempty if and only if $d \geq 8$ and $d \equiv 0,2$ $[6]$. Moreover, Hassett described how, in certain cases, one can associate to a cubic fourfold a K3 surface. More precisely, there exists a polarized K3 surface $S$ of degree $d$ such that $K_d^{\perp} \subset H^4(X,\Z)$ is Hodge-isometric to $H^2_{prim}(S,\Z)(-1)$ if and only if $d$ is not divisible by 4, 9, or any odd prime number $p \equiv 2\ [3]$. It is now conjectured that the existence of such a K3 is in relation with rationality of the corresponding cubic fourfold. According to the conjecture, supported by Hassett and Kuznestov's work (\cite{Ha1}, \cite{Ha2}, \cite{Kuznetsov_2009}), a cubic fourfold is  rational if and only if it has an associated K3 surface.  Moreover, for infinitely many values of $d$, the Fano variety of lines on the cubic fourfold $F(X)$ is isomorphic to the Hilbert scheme of length two subschemes $S^{[2]}$ of the associated K3 surface $S$. This holds if $d=2(n^2+n+1)$ for an integer $n\geq2$.

\medskip

\subsection{The divisor $\mathcal{C}_{14}$}

In this note, we will need to consider a particular Hassett divisor, that is $\mathcal{C}_{14}$. This can be considered in three  different ways.\par
(i) as  the closure of the locus of Pfaffian cubic fourfolds, i.e. cubics defined by the Pfaffian of a $6\times 6$ alternating matrix of linear forms.\par
(ii) as the closure of the locus of cubics containing a del Pezzo quintic surface;\par
(iii) as the closure of the locus of cubics containing a quartic scroll.\par
\noindent  The three descriptions are strictly related (see \cite{Beau,MR3968870,Ha3}). All cubics in $\mathcal{C}_{14}$ have an associated genus 8 K3 surface, and they are all rational \cite{MR3968870}.

\medskip

\subsection{Intersections of Hassett divisors}

Using lattice theoretical computations, one can study the intersection of Hassett divisors $\mathcal{C}_d \subset \mathcal{C}$ (see \cite{MR3238111}, \cite{awada2019rational}, \cite{2019arXiv190501936Y}). 
In \cite[Theorem 3.1]{2019arXiv190501936Y} it is proved that any two Hassett divisors intersect i.e $\mathcal{C}_{d_1} \cap \mathcal{C}_{d_2} \ne \emptyset$ for any integers $d_1, d_2 \equiv 0,2$ $[6]$. In order to study further the intersection of divisors, in \cite{awada2020rational}, the first named author proved that - by choosing appropriately the discriminants of the divisors - one can intersect up to 20 divisors $\mathcal{C}_d \in \mathcal{C}$ and have nonempty intersection.

\smallskip
\begin{thm}\label{intdiv}\cite{awada2020rational}
For $3\leq n \leq 20$,  
 
\begin{equation}\label{family}
     \displaystyle{\bigcap_{k=1}^{n}} \mathcal{C}_{d_k} \ne \emptyset,
\end{equation}
 
for $d_k \geq 8, d_k \equiv 0,2[6]$ and $d_3,..,d_n= 6 \displaystyle{\prod_i} p_i^2$ or $6 \displaystyle{\prod_i} p_i^2+2$ with $p_i$ a prime number.
\end{thm}

A couple of observations are in order. The first remark is that the first two discriminants $d_1$ and $d_2$ can be chosen completely arbitrarily. The second is the following. A very general cubic fourfold $X$ in $\mathcal{C}$ has $CH_2(X)=\mathbb{Z}$, and each time that we take the intersection with a new divisor $\mathcal{C}_{d_i}$ we know that we are adding an algebraic cycle inside $A(X)$. This means that the generic cubic fourfold in $\displaystyle{\bigcap_{k=1}^{n}} \mathcal{C}_{d_k} \ne \emptyset$ has $rk(CH_2(X))=n+1$. If one of the discriminants $d_i$ defines a divisor of cubics with associated K3 surface, it means that all the cubic fourfolds inside the intersection \ref{family} have associated K3 surfaces and, by the definition itself, the generic associated K3 surface has Neron-Severi (NS in what follows) rank equal to $n$.

\medskip

This has allowed the first named author to produce examples of rational cubic fourfolds with associated K3 surfaces of high NS rank (Picard number denoted by $\rho$) $\rho \geq 19$  (see \cite{shioda_inose_1977}, \cite{articlesarti}). These K3 surfaces have a Shioda-Inose structure  and have some particular features (see \cite{1984InMat..75..105M}).

\section{Chow motives of cubic fourfolds and K3 surfaces}


Let  $X \subset \P^5$ be a cubic 4-fold. The first result we need to recall from \cite{BP20} is the existence of a Chow-K\"unneth  decomposition for the motive of $X$. Namely we have

\begin{equation} \label{CK} h(X) = \un \oplus \L \oplus \L^{\rho_2(X)}\oplus t(X)\oplus \L^3 \oplus \L^4, \end{equation}

where $\rho_2(X) = \rank CH_2(X)$,  with $CH_2(X)\subset H^4(X,\Z)$, and $t(X)$ is the transcendental motive of $X$, i.e. $H^*(t(X)) =H^4_{tr}(X,\Q)$. 

\medskip

Let now $X$ be a special cubic fourfold contained in a divisor $\mathcal{C}_d$. In \cite{Bu}, B\"ulles shows that for certain values of $d$, there exists a K3 surface $S$ such that 

\begin{equation}\label{motk3}
t(X) \simeq t_2(S)(1).
\end{equation}

Here $t_2(S)$ is the transcendental motive of $S$ , i.e.

$$h(S)=\un \oplus \L^{\rho(S)}\oplus t_2(S) \oplus \L^2,$$

where $\rho(S)$ is the rank of the Neron-Severi group $NS(S)$. More precisely, the isomorphism (\ref{motk3}) holds whenever $d$ satisfies the following numerical condition

$$(***) : \exists  f,g \in \mathbb{Z}  \  with \  g\vert (2n^2+2n +2)   \  n\in \N  \ and \   d =f^2g.$$

Therefore, in this case, $h(X) \in \sM^{Ab}_{rat}$ if and only if $h(S)\in \sM^{Ab}_{rat}$. Note that an isomorphism $t(X)\simeq t_2(S)(1)$ can never hold if $X$ is not special, i.e. if $\rho_2(X)=1$, see \cite[Prop.3.4]{BP20}. Remark moreover that $d=14$  satisfies (***) with  $f=1,\ n=2$ and $g=14$.
 
\medskip

On the other hand, finite dimensionality of motives of K3 surfaces has been addressed in \cite{Ped}. In particular the following is proved

\begin{thm}\label{k3motive}
Let $S$ be a smooth complex projective K3 surface with $\rho(S) = 19, 20$.
Then the motive $h(X) \in \mathcal{M}_{rat} (C)$ is finite dimensional and lies in the subcategory
$\mathcal{M}_{rat}^{Ab} (C)$. 
\end{thm}

\section{Main Theorem}

In this section we will give the proof of the following theorem.

\begin{thm}\label{main2}
Every Hassett divisor $\mathcal{C}_d$ contains a one dimensional family of cubic fourfolds, whose Chow motive is finite dimensional and Abelian.
\end{thm}

\begin{proof}
Let us consider $\mathcal{C}_d\subset \mathcal{C}$ any divisor of special cubic fourfolds. By Theorem \ref{intdiv}, we can chose appropriately 17 divisors $\mathcal{C}_{d_1},\dots \mathcal{C}_{d_{17}}$ such that

$$\mathcal{F}:=\mathcal{C}_d\cap \mathcal{C}_{14}\cap (\bigcap_{k=1}^{17})\mathcal{C}_{d_k} \neq \emptyset.$$

In particular, by the proof of Theorem \ref{intdiv} in \cite{awada2020rational}, one sees that $\mathcal{F}$ is a dimension one algebraic subvariety of $\mathcal{C}_d$, that is a one dimensional family of cubic fourfolds. On the other hand, by construction, the family $\mathcal{F}$ is also contained in $\mathcal{C}_{14}$, hence all the cubics in $\mathcal{F}$ have an associated K3 surface. Moreover, by the results of Section 3 of \cite{awada2020rational} we observe that cubic fourfolds in $\mathcal{F}$ have associated K3 surfaces with Néron-Severi rank $\rho(S)\geq 19$. More precisely: the generic cubic fourfold in $\mathcal{F}$ has an associated K3 surface with $\rho(S)=19$; the intersection points of $\mathcal{F}$ with further Hassett divisors represent cubic fourfolds whose associated K3 surface has $\rho(S)=20$. By Thm. \ref{k3motive} the Chow motives of these K3 surfaces are finite dimensional and Abelian. Now we need to evoke the isomorphism of Eq. \ref{motk3}. The divisor $\mathcal{C}_{14}$ is among those whose cubic fourfolds have Chow motives that decomposes as follows

$$h(X) =h(X) = \un \oplus \L \oplus \L^{\rho_2(X)}\oplus t_2(S)(1)\oplus \L^3 \oplus \L^4$$

where $t_2(S)$ is the transcendental part of the motive $h(S)$ of the associated K3 surface. Now, if the motive of the associated K3 surface is finite dimensional or Abelian, then also the motive of the cubic fourfold has the same property. 
This means in turn that, by Thm. \ref{k3motive}, all the cubics in $\mathcal{F}$ have finite dimensional and Abelian Chow motive, since the associated K3 surfaces have Neron-Severi group of rank bigger or equal to 19.
\end{proof}

\begin{rk}
Let us point out that, since all the family $\mathcal{F}$ is contained in $\mathcal{C}_{14}$, then, by results of \cite{MR3968870}, all cubic fourfolds in $\mathcal{F}$ are rational.
\end{rk}

In the proof of Thm. \ref{main2}, we chose $\mathcal{C}_{14}$ for simplicity, since it is the first of the series that has $d$ verifying condition $(***)$, and where cubics have associated K3s. Any other divisor $\mathcal{C}_d$ obeying $(***)$, such that cubics in $\mathcal{C}_d$ have associated K3s, would have worked. Hence we can say even more.

\begin{cor}
Every Hassett divisor $\mathcal{C}_d$ contains a countable infinity of one dimensional families of cubic fourfolds, whose Chow motive is finite dimensional and Abelian.
\end{cor}

\begin{proof}
Just consider, in $(***)$, $f=1$ and $g=2n^2+2n+2$, for any $n\in \mathbb{N}$. This gives an infinite series of values of $d$, such that cubics in $\mathcal{C}_d$ have associated K3 surfaces.
\end{proof}

\subsection{Density of cubic fourfolds with abelian motive}

More generally, let $\mathcal{G}_d$ the moduli space of polarized K3 surfaces of degree d. This is a quasi-projective 19-dimensional algebraic variety.

\begin{thm}
Let $d$ be $d$ not divisible by 4, 9, or any odd prime number $p \equiv 2\ [3]$. Then there exists a countable, dense (in the complex topology) number of points in a non-empty Zariski open subset inside $\mathcal{C}_d$ such that the corresponding fourfolds have finite dimensional Chow motive.
\end{thm}

\begin{proof}
Let $X\in \mathcal{C}_d$, for $d$ in the range of the claim here above. That is: $X$ has one (or two, see \cite{Ha1}) associated polarized K3 surface $S_X$ in $\mathcal{G}_{\frac{d+2}{2}}$. Then the map

\begin{eqnarray}\label{asso}
\mathcal{G}_{\frac{d+2}{2}} & \to & \mathcal{C}_d;\\
S_X & \mapsto & X;
\end{eqnarray}

is rational and dominant. Hence, if $d$ is in the range here above, there exists an open set $\mathcal{U}_d$ of $\mathcal{C}_d$ such that for every $X\in \mathcal{U}_d$ there exists a K3 surface $S_X$ of degree $d$ associated to $X$. We observe also that singular K3 surfaces form a (countable) subset of the moduli space $\mathcal{G}_{\frac{d+2}{2}}$ which is dense in the complex topology. The proof of this fact goes along the same lines as the proof of the density of all K3 surfaces in the period domain (see  \cite[Corollary VIII.8.5]{ccs}). By the dominance of the map in $(\ref{asso})$, this directly implies the claim.
\end{proof}

\subsection{Some remarks on Hyperk\"ahler varieties}

In this last part of the paper, we draw consequences on Abelianity and finite dimensionality of the motive of some Hyperk\"ahler varieties related to cubic fourfolds from the results of the preceding section.

\medskip

Notably, we will consider $F(X)$  the Fano variety of lines and $L(X)$, the 8-fold constructed in \cite{LLSvS} from the space of twisted cubic curves on a cubic fourfold not containing a plane. The 4-dimensional $F(X)$ is in general deformation equivalent to the Hilbert scheme  $S^{[2]}$, with $S$ a K3 surface, while $L(X)$ is deformation equivalent to $S^{[4]}$. 
For every even complex dimension  there are two known deformations classes of irreducible holomorphic symplectic varieties: the Hilbert scheme $S^{[n]}$ of n-points on a K3 surface $S$ and the generalized Kummer. A generalized Kummer variety $X$ is of the form $X=K^n(A)=a^{-1}(0)$, where $A$ is an abelian surface and $a:  A^{[n+1]} \to A$ is the Albanese map. 
In dimension 10 there is also an example, usually referred as OG10,  discovered by O'Grady. The Hyperk\"ahler variety OG10 is not deformation equivalent to $S^{[5]}$.

\medskip

Let  $\sM_A(\C)$ be  the category of André motives which is obtained from the category of homological motives $\sM_{hom}(\C)$ by formally adjoining the Lefschetz involutions $L^{d-i}: H^i(X) \to H^{2d-i}(X)$, where $L^{d-i}$ is induced by the hyperplane section. By the Standard Conjecture $B(X)$, for every $i\le d$ there exists an algebraic correspondence inducing the isomorphism $H^{2d-i}(X)\to H(X)$ inverse to $L^{d-i}$. Therefore, under $B(X)$ the category of André motives coincides with $\sM_{hom}(\C)$.   The André motive of a K3  surface $S$  and of  a cubic fourfold $X$ belong to the  full subcategory  $\sM^{Ab}_{A}(\C) $ generated by the motives of abelian varieties, see [An,10.2.4.1].

\medskip

In [Sc] it is proved that the André motive of a Hyperk\"ahler variety which is  deformation equivalent  to  $S^{[n]}$ lies in $\sM^{Ab}_A$.  Soldatenkov [So] proves that  if $X_1$ and $X_2$ are  deformation equivalent projective Hyperk\"ahler manifolds then the André  motive of $X_1$  is abelian if and only if the André motive of $X_2$ is abelian. In a recent preprint (see \cite{FFZ})  it is proved that also the André motive of OG10 lies in $\sM^{Ab}_{A}(\C)$. Therefore the André motives of all the known deformation classes of Hyperk\"ahler varieties lie in $\sM^{Ab}_A$. These results suggest the following conjecture

\begin{conj}  The motive of a Hyperk\"ahler manifold  is of Abelian type in $\sM_{rat}(\C)$.\end{conj}

By \cite[Sect. 6]{decamiglio}, the Hilbert scheme $S^{[n]}$ of a K3 surface with finite dimensional (or Abelian) motive has finite dimensional (or Abelian) motive. Now, recall that our family $\mathcal{F}$ of cubic fourfolds from Thm. \ref{main2} entirely lies in $\mathcal{C}_{14}$. Moreover, for all cubics $X$ in $\mathcal{C}_{14}$, we have an isomorphism $F(X)\cong S^{[2]}$ \cite{Ha1}, where $S$ is the associated K3. Hence it is straightforward to check that we have the following.

\begin{prop}\label{fano}
All Hyperk\"ahler fourfolds $F(X)$, $X\in \mathcal{F}$, have finitely generated
and Abelian Chow motive.
\end{prop}

\begin{rk}
Once again, we can play the same game as before by taking $\mathcal{C}_{\frac{2n^2+2n+2}{a^2}}$, $n,a\in \Z$, instead of $\mathcal{C}_{14}$. By \cite[Thm. 2]{Add16}, having a $d$ of this shape is equivalent to having a birational equivalence between $F(X)$ and the Hilbert square of the associated K3. Since birational Hyperk\"ahler varieties have isomorphic Chow motives, everything runs the same way, and we have a countably infinite set of families of Fano varieties $F(X)$ with finite dimensional and Abelian motive, whose cubic fourfolds all lie in a fixed $\mathcal{C}_d$.
\end{rk}

On the other hand, let us now consider the Hyperk\"ahler 8fold $L(X)$. In order to define properly $L(X)$ we need to assume that $X$ does not contain a plane, i.e. $X\not\in\mathcal{C}_8$. Then, the analogue of Prop. \ref{fano} is the following.

\begin{prop}\label{8fold}
All Hyperk\"ahler 8folds $L(X)$, $X\in \mathcal{F}$, have finitely generated and Abelian Chow motive.
\end{prop}

\begin{proof}


Let $X\in \mathcal{C}_d$ be a cubic fourfold not containing a plane, such that $S$ is its associated K3 surface. In \cite[Thm. 3]{AddGio} the authors show that the 8fold $L(X)$ is birational to $S^{[4]}$ if and only if

$$(***')\ \ \ d=\frac{6n^2+6n+2}{a^2},\ \ n,a\in\Z.$$

The first integer of the list is once again 14, hence for all the cubic fourfolds of our family $\mathcal{F}$ we have $L(X)\stackrel{birat}{\cong} S^{[4]}$. Since birational Hyperk\"ahler varieties have isomorphic Chow motives, the results from \cite{decamiglio} complete the proof, and we have a one-dimensional family of $L(X)$ with finite dimensional and Abelian motive for all $\mathcal{C}_d$.
\end{proof}

\begin{rk}
If, instead of taking $\mathcal{C}_{14}$ in the proof of Prop \ref{fano} and \ref{8fold}, we take $d=182$, we observe that this verifies both condition $(***)$ (with $f=1$) and $(***')$. Then the intersection

\begin{equation}\label{both}
\mathcal{F}:=\mathcal{C}_d\cap \mathcal{C}_{182}\cap (\bigcap_{k=1}^{17})\mathcal{C}_{d_k}  \neq \emptyset.
\end{equation}

defines a one dimensional family, inside any $\mathcal{C}_d$ of cubic fourfolds s.t. the corresponding Fano varieties and LLSvS 8folds have finite dimensional and Abelian motive, too.
\end{rk}

\bibliography{bib_tocho}
\bibliographystyle{abbrv}

\end{document}